\documentclass[11pt,reqno]{amsart}
\usepackage{hyperref}

\usepackage{enumerate}

\usepackage[utf8]{inputenc}
\usepackage[T1]{fontenc}

\usepackage[mathscr]{eucal}
\usepackage{amsmath,amsfonts,amsthm,amssymb}
\usepackage[all]{xy}

\usepackage[normalem]{ulem}

\usepackage[a4paper, margin=1.2 in, top=1.3 in, bottom= 1.3 in]{geometry}

\usepackage{mathtools}
\usepackage[capitalize]{cleveref}
\usepackage{color}

\newtheorem{lemma}{Lemma}
\newtheorem{theorem}[lemma]{Theorem}
\newtheorem*{theorem*}{Theorem}

\newtheorem{question}{Question}

\newtheorem*{proposition*}{Proposition}

\newcounter{MainTheoremCounter}

\newtheorem{Maintheorem}[MainTheoremCounter]{Theorem}

\newtheoremstyle{newdefinitionstyle}{2mm}{2mm}{}{}{\bfseries}{.}{.5em}{}

\theoremstyle{newdefinitionstyle}
\newtheorem*{definition*}{Definition}
\newtheorem*{conjecture*}{Conjecture}
\newtheorem{remark}[lemma]{Remark}
\newtheorem*{remarks*}{Remarks}
\newtheorem*{claim*}{Claim}

\newtheorem{example}[lemma]{Example}

\newcommand{\C}{{\mathbb C}}

\newcommand{\N}{{\mathbb N}}
\renewcommand{\P}{{\mathbb P}}
\newcommand{\Q}{{\mathbb Q}}
\newcommand{\R}{{\mathbb R}}
\newcommand{\T}{{\mathbb T}}
\newcommand{\Z}{{\mathbb Z}}

\definecolor{orange}{rgb}{1,0.5,0}

\newcommand{\stkout}[1]{\ifmmode\text{\sout{\ensuremath{#1}}}\else\sout{#1}\fi}

\author[A.\ Koutsogiannis]{Andreas Koutsogiannis}
\address{Department of Mathematics\\ The Ohio State University\\
Columbus, OH, USA}
\email{koutsogiannis.1@osu.edu}

\author[A.\ N.\ Le]{Anh N. Le}
\address{Department of Mathematics\\ Northwestern University\\
Evanston, IL, USA}
\email{anhle@math.northwestern.edu}

\author[J.\ Moreira]{Joel Moreira}
\address{Mathematics Institute\\ University of Warwick\\
Coventry, UK}
\email{joel.moreira@warwick.ac.uk}

\author[F.\ K.\ Richter]{Florian K. Richter}
\address{Department of Mathematics\\ Northwestern University\\
Evanston, IL, USA}
\email{fkr@northwestern.edu}

\begin{document}

\begin{abstract}
 Let $T$ be a measure preserving $\Z^\ell$-action on the probability space $(X,{\mathcal B},\mu),$ $q_1,\dots,q_m\colon\R\to\R^\ell$ vector polynomials, and $f_0,\dots,f_m\in L^\infty(X)$. For any $\epsilon > 0$ and multicorrelation sequences of the form $\displaystyle\alpha(n)=\int_Xf_0\cdot T^{ \lfloor q_1(n) \rfloor }f_1\cdots T^{ \lfloor q_m(n) \rfloor }f_m\;d\mu$
we show that there exists a nilsequence $\psi$  for which
$\displaystyle\lim_{N - M \to \infty} \frac{1}{N-M} \sum_{n=M}^{N-1} |\alpha(n) - \psi(n)| \leq \epsilon$
and
$\displaystyle\lim_{N \to \infty} \frac{1}{\pi(N)} \sum_{p \in \P\cap[1,N]} |\alpha(p) - \psi(p)| \leq \epsilon.$
This result simultaneously generalizes previous results of Frantzikinakis \cite{Frantzikinakis15b} and the authors \cite{Koutsogiannis_multi_2018,Le_Moreira_Richter_2020}.

\end{abstract}



\title[]{Structure of multicorrelation sequences with integer part polynomial iterates along primes }

\vspace*{-1cm}
\maketitle

\section{Introduction and main result}
\label{sec_intro}

\noindent Since Furstenberg's ergodic theoretic proof of Szemer\'edi's theorem \cite{Furstenberg77}, there has been much interest in understanding the structure of \emph{multicorrelation sequences}, i.e., sequences of the form
\begin{equation}\label{eq_BHKmulticorrelation}
    \alpha(n)=\int_X f_0\cdot T^nf_1\cdots T^{kn}f_k\;d\mu,\;\;n\in\N,
\end{equation}
where $(X,{\mathcal B},\mu,T)$ is a measure preserving system and $f_0,\dots,f_k\in L^\infty(X)$.
The first to provide deeper insight into the algebraic structure of such sequences were Bergelson, Host, and Kra, who showed in \cite{Bergelson_Host_Kra05} that if the system $(X, \mu, T)$ is ergodic then for any multicorrelation sequence $\alpha$ as in \eqref{eq_BHKmulticorrelation} there exists a uniform limit of $k$-step nilsequences  $\phi$ such that
\begin{equation}\label{eq_BHK}
    \lim_{N-M\to\infty}\frac1{N-M}\sum_{n=M}^{N-1}\big|\alpha(n)-\phi(n)\big|=0.
\end{equation}

Here, a \emph{$k$-step nilsequence} is a sequence of the form $\psi(n)=F(g^nx),$ $n\in\N,$ where $F$ is a continuous function on a $k$-step nilmanifold $X=G/\Gamma$,\footnote{A \emph{$k$-step nilmanifold} is a homogeneous space $X=G/\Gamma$, where $G$ is a $k$-step nilpotent Lie group and $\Gamma$ is a discrete and co-compact subgroup of $G$.} $g\in G,$ $x\in X$. A \emph{uniform limit of $k$-step nilsequences} is a sequence $\phi$ such that for every $\epsilon>0$ there exists a $k$-step nilsequence $\psi$ with $\operatorname{sup}_{n\in\N}|\phi(n)-\psi(n)|\leq\epsilon$.

\medskip

Later, Leibman extended the result of Bergelson, Host and Kra to polynomial iterates in \cite{Leibman10}, and removed the ergodicity assumption in \cite{Leibman15}.
Another extension was obtained by the second author in \cite{Le17}, and independently by Tao and Ter\"av\"ainen in \cite{Tao_Teravainen_17}, answering a question raised in \cite{Frantzikinakis2016}. There, it was shown that in addition to \eqref{eq_BHK} one also has
\begin{equation}\label{eq_BHKprimes}
    \lim_{N\to\infty}\frac1{\pi(N)}\sum_{p\in\P\cap[1,N]}\big|\alpha(p)-\phi(p)\big|=0,
\end{equation}
where $\P$ denotes the set of prime numbers, $[1,N]\coloneqq \{1,\dots,N\}$, and $\pi(N)\coloneqq |\P\cap[1,N]|.$

\medskip

The proofs of all the aforementioned results depend crucially on the 
structure theory of Host and Kra, who established in \cite{Host_Kra05} that the building blocks of the factors that control multiple ergodic averages are nilsystems. 
Since the analogous factors for $\Z^\ell$-actions are unknown, extending the results above from $\Z$-actions to $\Z^\ell$-actions proved to be a challenge. Nevertheless, in \cite{Frantzikinakis15b} Frantzikinakis concocted a different approach and gave a description of the structure of multicorrelation sequences of $\Z^\ell$-actions, which we now explain.

\medskip

Henceforth, let $\ell\in\N$ and let $T$ be a measure preserving $\Z^\ell$-action on a probability space $(X,{\mathcal B},\mu)$.
The system $(X,{\mathcal B},\mu,T)$ gives rise to a more general class of multicorrelation sequences,
\begin{equation}
\label{eq_multicorrelation}
  \alpha(n)=\int_Xf_0\cdot T^{ q_1(n)}f_1\cdots T^{ q_m(n)}f_m\;d\mu, \;\;n\in\N,
\end{equation}
where $q_1,\dots,q_m\colon\Z\to\Z^\ell$ are integer-valued vector polynomials and $f_0,\dots,f_m\in L^\infty(X)$.
Note that \eqref{eq_BHKmulticorrelation} corresponds to the special case of \eqref{eq_multicorrelation} when $\ell=1$ and $q_i(n)=in$.
Frantzikinakis showed in  \cite{Frantzikinakis15b} that for every $\alpha$ as in \eqref{eq_multicorrelation} and every $\epsilon>0$ there exists a $k$-step nilsequence $\psi$ such that
    \begin{equation}\label{eq_weakBHK}
    \lim_{N-M\to\infty}\frac1{N-M}\sum_{n=M}^{N-1}\big|\alpha(n)-\psi(n)\big|\leq\epsilon,
    \end{equation}
where $k$ only depends on $\ell$, $m$, and the maximal degree among the polynomials  $q_1,\ldots,q_m$.
Moreover, in the special case where each polynomial iterate is linear, it was proved in \cite{Frantzikinakis15b} that one can take $k=m$.
It is still an open question whether in \eqref{eq_weakBHK} one can replace $\epsilon$ with $0$ after replacing the nilsequence $\psi$ with a uniform limit of such sequences (see \cref{question_zeroepsilon} in \cref{sec:open_questions}).

For $x\in\R$ we denote by $\lfloor x\rfloor$ the largest integer which is smaller or equal to $x$, while for $x=(x_1,\dots,x_\ell)\in\R^\ell$ we let $\lfloor x\rfloor\coloneqq(\lfloor x_1\rfloor,\dots,\lfloor x_\ell\rfloor)$.
In \cite{Koutsogiannis_multi_2018}, the first author extended Frantzikinakis' results to all multicorrelation sequences of the form
\begin{equation}\label{eq:alpha}
\alpha(n)=\int_Xf_0\cdot T^{\lfloor q_1(n)\rfloor}f_1\cdots T^{\lfloor q_m(n)\rfloor} f_m\;d\mu,\;\;n\in\N,
\end{equation}
where $q_1,\dots,q_m\colon\R\to\R^\ell$ are real-valued vector polynomials.

More recently, the last three authors showed that the conclusion of Frantzikinakis' result also holds along the primes:

\begin{theorem}[{\cite[Theorems A and B]{Le_Moreira_Richter_2020}}]\label{thm_B}For every $\ell, m,d\in \mathbb{N}$ there exists $k\in\N$ with the following property.
For any polynomials $q_1,\ldots,q_m\colon\Z\to\Z^\ell$ with degree at most $d$, measure preserving $\Z^\ell$-action $T$ on a probability space $(X,\mathcal{B},\mu)$, functions $f_0,f_1,\ldots,f_m\in L^\infty(X)$, $\varepsilon>0,$ $r\in \mathbb{N}$ and $s\in \mathbb{Z},$ letting $\alpha$ be as in \eqref{eq_multicorrelation},
there exists a $k$-step nilsequence $\psi$  satisfying \eqref{eq_weakBHK}
and
\begin{equation}\label{primes_ap}
\lim_{N\to\infty}\frac{1}{\pi(N)}\sum_{p\in\mathbb{P}\cap[1,N]}|\alpha(rp+s)-\psi(rp+s)|\leq \varepsilon.
\end{equation}
In the special case $d=1$ one can choose $k=m$.
\end{theorem}

Our main theorem simultaneously generalizes the main results from \cite{Koutsogiannis_multi_2018} and \cite{Le_Moreira_Richter_2020}.

\begin{Maintheorem}\label{thm:main}
For every $\ell, m,d\in \mathbb{N}$ there exists $k\in\N$ with the following property.
For any polynomials $q_1,\ldots,q_m\colon\R\to\R^\ell$ with degree at most $d$, any measure preserving $\Z^\ell$-action $T$ on a probability space $(X,\mathcal{B},\mu)$, functions $f_0,f_1,\ldots,f_m\in L^\infty(X)$, $\varepsilon>0,$ $r\in \mathbb{N}$ and $s\in \mathbb{Z},$ letting $\alpha$ be as in \eqref{eq:alpha},
there exists a $k$-step nilsequence $\psi$ satisfying \eqref{eq_weakBHK} and \eqref{primes_ap}.
In the special case $d=1$ one can choose $k=m$.
\end{Maintheorem}

The proof of \cref{thm:main}, presented in the next section, follows closely the strategy implemented in \cite{Koutsogiannis_multi_2018}, but uses \cref{thm_B} instead of Walsh's theorem \cite{Walsh12} as a blackbox.

\begin{remark}\label{remark_commuting}
Both Theorems \ref{thm_B} and \ref{thm:main} are equivalent to seemingly stronger versions involving commuting actions.
We say that two actions $T_1$ and $T_2$ of a group $G$ commute if for every $g,h\in G$ we have $T_1^g\circ T_2^h=T_2^h\circ T_1^g$.
When $G$ is an abelian group, a collection of $m$ commuting $G$-actions $T_1,\dots,T_m$ can be identified with a single $G^m$-action $T$ via $T^{(g_1,\dots,g_m)}=T_1^{g_1}\cdots T_m^{g_m}$.
Using this observation, and the identification $(\Z^\ell)^m=\Z^{\ell m}$, one sees that, given commuting measure preserving $\Z^\ell$-actions $T_1,\dots,T_m$ in a probability space $(X,\mu,T)$, \cref{thm_B} holds when \eqref{eq_multicorrelation} is replaced by
\begin{equation}\label{eq_multicorrelationcommuting}
    \alpha(n)=\int_Xf_0\cdot T_1^{ q_1(n)}f_1\cdots T_m^{ q_m(n)}f_m\;d\mu,
\end{equation}
and \cref{thm:main} holds when \eqref{eq:alpha} is replaced by
$$\alpha(n)=\int_Xf_0\cdot T_1^{\lfloor q_1(n)\rfloor}f_1\cdots T_m^{\lfloor q_m(n)\rfloor} f_m\;d\mu.$$
\end{remark}

\begin{remark}
Let $\lceil x\rceil$ and $[x]$ denote the smallest integer $\geq x$ and the closest integer to $x$, respectively. Using the relations $\lceil x\rceil=-\lfloor -x\rfloor$ and $[x]=\lfloor x+1/2\rfloor$,
we see that \cref{thm:main} remains true if \eqref{eq:alpha} is replaced by
\begin{equation*}\label{eq_multicorrelation_R_mixandmatch}
\alpha(n)=\int_Xf_0\cdot T^{[ q_1(n)]_1}f_1\cdots T^{[q_m(n)]_m} f_m\;d\mu,\;\;n\in\N,
\end{equation*}
where $[x]_{i}=([x_1]_{i,1},\ldots,[x_\ell]_{i,\ell})$ and $[\cdot]_{i,1},\ldots,[\cdot]_{i,\ell}$ are any of $\lfloor\cdot\rfloor$, $\lceil\cdot\rceil$, or $[\cdot]$.
\end{remark}

\medskip

\paragraph{\textbf{Acknowledgements.}}~The fourth author is supported by the National Science Foundation under grant number DMS~1901453.

 \section{Proof of main result}\label{Intermediate}

We start by proving a theorem concerning flows, which stands halfway in between Theorems~\ref{thm_B} and ~\ref{thm:main}.
The idea behind this result is that for a real polynomial $q(x)=a_d x^d+\ldots+a_1 x+a_0\in \mathbb{R}[x]$  and  a measure presenting flow $(S^t)_{t\in \mathbb{R}}$ we can write $S^{q(n)}=\left(S^{a_d}\right)^{n^d}\cdots\left(S^{a_0}\right)^{1},$ an expression which can be handled by \cref{thm_B}.

\begin{theorem}\label{thm:flow}
    For every $\ell, m,d\in \mathbb{N}$ there exists $k\in\N$ with the following property.
For any polynomials $q_1,\ldots,q_m\colon\R\to\R^\ell$ with degree at most $d$, commuting measure preserving $\R^\ell$-actions $S_1,\dots,S_m$ on a probability space $(X,\mathcal{B},\mu)$, functions $f_0,f_1,\ldots,f_m\in L^\infty(X)$, $\varepsilon>0,$ $r\in \mathbb{N}$, and $s\in \mathbb{Z},$ 
letting \begin{equation}\label{eq_flow}
    \alpha(n) = \int_X f_0 \cdot S_{1}^{q_{ 1}(n)} f_1 \cdots S_m^{q_{m}(n)}f_m \, d \mu,
\end{equation}
there exists a $k$-step nilsequence $\psi$ satisfying \eqref{eq_weakBHK} and \eqref{primes_ap}. In the special case $d=1$ one can choose $k=m$.
\end{theorem}
\begin{proof}
    For each $i\in[1,m]$, let $q_i=(q_{i,1},\dots,q_{i,\ell})$ for some $q_{i,j}\in\R[x]$.
    Next, for each $j\in[1,\ell]$,  write $q_{i,j}(x) = \sum_{h = 0}^d a_{i,j,h} x^h$, where the $a_{i,j,h}$'s are real numbers.
    Also, for each $j\in[1,\ell]$, let $e_j$ be the $j$-th vector of the canonical basis of $\R^\ell$ and let 
    $T_{i,j,h}$ be the measure preserving transformation defined by $T_{i,j,h}=S_{i}^{a_{i,j,h}e_j}$.
    Next, let $T_{i,h}$ be the composition $T_{i,h}=T_{i,1,h}\cdots T_{i,\ell,h}$, let $T_i$ be the $\Z^{d+1}$-action defined by $T_i^{(n_0,\dots,n_d)}=T_{i,0}^{n_0}\cdots T_{i,d}^{n_d},$ and let  $q\colon\Z\to\Z^{d+1}$ be the polynomial $q(n)=(1,n,\dots,n^d)$. 
    
    With this setup, for each $i\in[1,m]$ and $n\in\N$, we have
    \[
        S_i^{q_i(n)}
        = 
        \prod_{j=1}^\ell S_{i}^{q_{i,j}(n)e_j}
        =
        \prod_{j=1}^\ell\prod_{h=0}^d T_{i,j,h}^{n^h}
        =
        \prod_{h=0}^dT_{i,h}^{n^h}
        =
        T_i^{q(n)}.  
    \]
    Since the $\R^\ell$-actions $S_1,\dots, S_m$ commute, so do the $\Z^{d+1}$-actions $T_1,\dots,T_m$. 
    This implies that the multicorrelation sequence $\alpha$ can be represented by an expression of the form \eqref{eq_multicorrelationcommuting}. The conclusion now follows directly from \cref{thm_B} and \cref{remark_commuting}.
\end{proof}

Next we need a result concerning the distribution of real polynomials.

\begin{lemma}
\label{lem:equidistributed}
    Let $q \in \R[x]$ be a non-constant real polynomial, $r \in \N$ and $s \in \Z$. Then, denoting by $\{\cdot\}$ the fractional part, we have
    \[
        \lim_{\delta \to 0^+} \lim_{N-M \to \infty} \frac1{N-M}\left|\Big\{n \in [M, N): \big\{q(n)\big\} \in [1 - \delta, 1) \Big\}\right| = 0,
    \]
    and
    \[
        \lim_{\delta \to 0^+} \lim_{N \to \infty} \frac1{\pi(N)}\left|\Big\{p \in \P \cap [1,N]: \big\{q(rp + s)\big\} \in [1 - \delta, 1)\Big\}\right| = 0.
    \]
\end{lemma}
\begin{proof}
    Let
    \[
        A(\delta) = \lim_{N-M \to \infty} \frac1{N-M}\left|\Big\{n \in [M, N): \big\{q(n)\big\} \in [1 - \delta, 1) \Big\}\right|,
    \]
    and
    \[
        B(\delta) = \lim_{N \to \infty} \frac1{\pi(N)}\left|\Big\{p \in \P \cap [1,N]: \big\{q(rp + s)\big\} \in [1 - \delta, 1)\Big\}\right|.
    \]
    If $q-q(0)$ has an irrational coefficient, then by Weyl's Uniform Distribution Theorem \cite{Weyl16} and Rhin's Theorem \cite{Rhin_1973} we have $
        A(\delta) = B(\delta) = \delta$
    which approach $0$ as $\delta \to 0^+$.

    Assume otherwise that $q \in \R[x]$ satisfies $q-q(0)\in\Q[x]$,  say $q(x) = q(0)+b^{-1}\sum_{j = 1}^{\ell} a_{j} x^j$ where $b\in \N,$  $a_j \in \Z$ for $1 \leq j \leq \ell$, and $q(0)\in\R$. It follows that for all $n \in \N$,
    \[
        q(n)-q(0) \bmod 1 \in \left\{ 0, \frac{1}{b}, \frac{2}{b}, \ldots, \frac{b-1}{b} \right\}.
    \]
    In particular, the fractional part $\{q(n)\}$ takes only finitely many values.
    Therefore, if $\delta$ is small enough, for every $n \in \N$ we have $\{q(n)\} \notin [1 - \delta, 1)$ and hence $A(\delta) = B(\delta) = 0,$ which implies the desired conclusion. 
\end{proof}

For the proof of \cref{thm:main} we adapt arguments from \cite{Koutsogiannis_primes_2018, Koutsogiannis_multi_2018}, i.e., we use a multidimensional suspension flow to approximate $\alpha$ by a multicorrelation sequence of the form \eqref{eq_flow}. The arising error consists of terms of the form $1_{\{n\in\N:\;q(n)\in [1-\delta,1)\}}$ that can be controlled by \cref{lem:equidistributed}.

\begin{proof}[Proof of \cref{thm:main}]
Given $\ell,m,d\in\N$, let
    $k$ be as guaranteed by \cref{thm:flow}.
    Let $q_1,\dots,q_m$, $T$, $f_0,\dots,f_m$, $\epsilon>0$, $r\in\N$, $s\in\Z$ and $\alpha$ be as in the statement.
    By multiplying each function by a constant if needed, we can assume without loss of generality that $\lVert f_i \rVert_{\infty} \leq 1$ for each $i\in[1,m]$.

    We start by considering a multidimensional suspension flow with a constant $1$ ceiling function.
    More precisely, let $Y\coloneqq X\times[0,1)^{m\times\ell}$ and $\nu=\mu\otimes\lambda$, where $\lambda$ denotes the Lebesgue measure on $[0,1)^{m\times\ell}$.
    For each $i\in[1,m]$ define the measure preserving $\R^\ell$-action $S_i$ on $(Y,\nu)$ as follows: for any $t\in\R^\ell$ and  $(x;b_1,\dots,b_m)\in Y=X\times\big([0,1)^\ell\big)^m,$ let
    $$S_i^t(x;b_1,\dots,b_m)
    \coloneqq
    \big(T^{\lfloor b_i+t\rfloor}x; b_1,\dots,b_{i-1},\{b_i+t\},b_{i+1},\dots,b_m\big),$$
    where $\{u\}\coloneqq u-\lfloor u\rfloor$ for any $u\in\R^\ell$.
    Observe that the actions $S_1,\dots,S_m$ commute.

    Let $\pi\colon Y\to X$ be the natural projection and $\delta>0$ a small parameter to be determined later.
    For each $i\in[1,m]$ let $\hat f_i\in L^\infty(Y)$ be the composition $\hat f_i\coloneqq f_i\circ\pi$, and $\hat f_0\coloneqq 1_{X\times[0,\delta]^{m\times\ell}}\cdot f_0\circ\pi$.
    Define
    \[
        \tilde{\alpha}(n) = \int_Y \hat{f}_0 \cdot S_1^{q_1(n)} \hat{f}_1 \cdots S_m^{q_m(n)} \hat{f}_m \; d\nu.
    \]
    By \cref{thm:flow} there exists a $k$-step nilsequence $\tilde\psi$ such that
    \begin{equation}\label{eq_proof5}
        \lim_{N-M \to \infty} \frac{1}{N-M} \sum_{n=M}^{N-1} |\tilde{\alpha}(n) - \tilde\psi(n)| \leq \delta^{\ell m}\epsilon/2,
    \end{equation}
    and
    \begin{equation}\label{eq_proof6}
        \lim_{N \to \infty} \frac{1}{\pi(N)} \sum_{p \in \P \cap [1,N]} |\tilde{\alpha}(rp + s) - \tilde\psi(rp + s) | \leq \delta^{\ell m}\epsilon/2.
    \end{equation}
    On the other hand,
    $$\tilde\alpha(n)
    =
    \int_{[0,\delta]^{\ell m}}
    \int_X f_0(x)
    f_1\Big(T^{\lfloor q_1(n)+b_1\rfloor}x\Big)\cdots f_m\Big(T^{\lfloor q_m(n)+b_m\rfloor}x\Big)\,d\mu(x)\,d\lambda(b_1,\dots,b_m),$$
    which implies
    \begin{equation}\label{eq_alphacomparison}
        \alpha(n) - \frac{ \tilde{\alpha}(n)}{\delta^{\ell m}} =
         \frac1{\delta^{\ell m}}\int_{[0, \delta]^{\ell m}} \int_X f_0(x) \left( \prod_{i=1}^{m} f_i\Big(T^{\lfloor q_i(n) \rfloor} x\Big) - \prod_{i=1}^{m} f_i\Big(T^{\lfloor q_i(n) + b_i \rfloor} x\Big) \right)\, d \mu \, d \lambda.
    \end{equation}
    In particular, it follows from \eqref{eq_alphacomparison} that $|\alpha(n)-\delta^{-\ell m}\tilde\alpha(n)|\leq2$ for all $n\in\N$.
    If $b_i\in[0,\delta]^\ell$ and $\big\{q_i(n)\big\}\in[0,1-\delta)^\ell$ then $\lfloor q_i(n)+b_i\rfloor=\lfloor q_i(n)\rfloor$.
    Therefore \eqref{eq_alphacomparison} also implies that $\alpha(n)=\delta^{-\ell m}\tilde\alpha(n)$ whenever
    $$n\notin \Big\{n\in\N:\big\{q_i(n)\big\}\in[1-\delta,1)^\ell\text{ for some }i\in[1,m]\Big\}.$$
    In view of \cref{lem:equidistributed}, by choosing $\delta$ small enough, we have
    \begin{equation}\label{eq_proof7}\lim_{N-M \to \infty} \frac{1}{N-M} \sum_{n=M}^{N-1} \left|\alpha(n) -  \delta^{-\ell m} \tilde{\alpha}(n) \right|<\frac\epsilon2\end{equation}
    and
    \begin{equation}\label{eq_proof8}
         \lim_{N \to \infty} \frac{1}{\pi(N)} \sum_{p \in \P \cap [1,N]} \left|\alpha(rp + s) -  \delta^{-\ell m} \tilde{\alpha}(rp + s) \right|<\frac\epsilon2.
    \end{equation}
     Letting $\psi=\delta^{-m\ell}\tilde\psi$ and combining \eqref{eq_proof5} with \eqref{eq_proof7} and \eqref{eq_proof6} with \eqref{eq_proof8} we obtain the desired conclusion.
\end{proof}

\begin{remark}
As it was already mentioned in \cref{sec_intro}, it is an open problem whether one can improve upon the approximation in Frantzikinakis' main result in \cite{Frantzikinakis15b} and \cref{thm_B} and take $\epsilon=0$ in \eqref{eq_weakBHK} and \eqref{primes_ap} (see \cref{question_zeroepsilon} below).
However, as the following example shows, in the case of \cref{thm:main} it is not possible to improve upon the approximation in that manner.
\end{remark}

\begin{example}\label{example}
Take $X=\T\coloneqq\R/\Z$, $T(x)=x+1/\sqrt{2}$, $q(n)=\sqrt{2}n$, $f_0(x)=e(x)$ and $f_1(x)=e(-x)$, where $e(x)\coloneqq e^{2\pi ix}$. Then we have
\begin{eqnarray*}
\alpha(n) & = &\int f_0\cdot T^{\lfloor q(n)\rfloor}f_1\, d\mu
    ~=~\int e(x) e\left(-x-\frac{1}{\sqrt{2}} \lfloor \sqrt{2}n\rfloor\right)\,dx \\
    & = & e \left( - \frac{1}{\sqrt{2}} \lfloor \sqrt{2} n \rfloor \right)
     ~=~ e\left(\frac{1}{\sqrt{2}}\{\sqrt{2}n\}\right).
\end{eqnarray*}
In particular, we can write $\alpha(n)$ as $F(T^n x_0)$ with $x_0 = 0 \in \T$ and $F(x) = e(\{x\}/\sqrt{2})$ for $x \in \T$. Assume for the sake of a contradiction that there exists a uniform limit of nilsequences $\phi$ for which
\begin{equation}
\label{eq:nil+null}
    \lim_{N \to \infty} \frac{1}{N} \sum_{n=1}^N |\alpha(n) - \phi(n)| = 0.
\end{equation}
By \cite[Lemma 18]{Host_Maass2007}, $\phi$ can be written as $\phi(n) = G(S^n y_0)$ for all $n \in \N,$ where $G$ is a continuous function on an inverse limit of nilsystems $(Y, S)$ and $y_0 \in Y$.

We claim that $\alpha(n) = \phi(n)$ for all $n \in \N$. 
If not, then there exists $\delta > 0$ and  $n_0 \in \N$ such that
\begin{equation}
\label{eq:not_eq}
    |\alpha(n_0) - \phi(n_0)| = |F(T^{n_0} x_0) - G(S^{n_0} y_0)| \geq \delta.
\end{equation}
Since the system $(X \times Y, T \times S)$ is the product of two distal systems, is a distal system itself. This implies that the point $(T^{n_0} x_0, S^{n_0} y_0)$ is uniformly recurrent, i.e., the sequence $(T^n x_0, S^n y_0)$ visits any neighborhood of $(T^{n_0} x_0, S^{n_0} y_0)$ in a syndetic set. 
This fact together with \eqref{eq:not_eq} and the fact that both the real and imaginary parts of $F$ are almost everywhere continuous and  semicontinuous imply that the set
\[
    \left\{n \in \N: |F(T^n x_0) - G(S^n y_0)| \geq \delta/2\right\}
\]
is syndetic, which contradicts \eqref{eq:nil+null}. Hence $\alpha(n) = \phi(n)$ for all $n \in \N$. However, by \cite[Proposition 4.2.5]{Haland92}, the sequence $\alpha$ is not a distal sequence; in particular, it is not a uniform limit of nilsequences, contradicting our assumption.
\end{example}

\section{Open questions}
\label{sec:open_questions}

We close this article with three open questions. \cref{thm:main} provides an approximation result of multicorrelation sequences along an integer polynomial of degree one, evaluated at primes. We can ask whether a similar result is true along other classes of sequences. 

\begin{question}\label{q1}
    Let $q \in \R[x]$ be a non-constant real polynomial, $c > 0$, and $p_n$ denote the n-th prime. Suppose $r_n = q(n), q(p_n), \lfloor n^c \rfloor$ or $\lfloor p_n^c \rfloor$ for $n \in \N$. Is it true that for any $\alpha$ as in \eqref{eq:alpha} and $\epsilon > 0$, there exists a nilsequence $\psi$ satisfying
    \[
        \lim_{N\to\infty} \frac{1}{N} \sum_{n = 1}^N |\alpha(r_n) - \psi(r_n)| \leq \epsilon?
   \] 
\end{question}

Variants of the following question have appeared several times in the literature, e.g., \cite[Remark after Theorem 1.1]{Frantzikinakis15b}, \cite[Problem 20]{Frantzikinakis2016}, \cite[Problem 1]{Frantzikinakis_Host_2018}, and \cite[Page 398]{Host_Kra_18}.

\begin{question}\label{question_zeroepsilon}
  Let $\alpha$ be as in \eqref{eq_multicorrelation}.
  Does there exist a  uniform limit of nilsequences $\phi$ such that
  \[
    \lim_{N-M\to\infty}\frac1{N-M}\sum_{n=M}^{N-1}\big|\alpha(n)-\phi(n)\big|=0?
  \]
\end{question}

As mentioned in \cref{example}, the answer to \cref{question_zeroepsilon} is negative when $\alpha$ is a multicorrelation sequence as in \eqref{eq:alpha}. Nevertheless, it makes sense to ask for the following modification of it.

\begin{question}
    Let $\alpha$ be as in \eqref{eq:alpha}.
  Does there exist a 
  uniform limit of \emph{Riemann integrable nilsequences} $\phi$ satisfying
  \[
    \lim_{N-M\to\infty}\frac1{N-M}\sum_{n=M}^{N-1}\big|\alpha(n)-\phi(n)\big|=0?
  \]
  Here we say that $\phi$ is a uniform limit of Riemann integrable nilsequences if for every $\epsilon>0$ there exists a nilmanifold $X=G/\Gamma$, a point $x\in X$, $g\in G$ and a Riemann integrable function\footnote{A function $F$ is \emph{Riemann integrable} on a nilmanifold  if its points of discontinuity is a null set with respect to the Haar measure.} $F\colon X\to\C$ such that $\sup_{n\in\N}|\phi(n)-F(g^nx)|<\epsilon$.
\end{question}

\medskip

\bibliographystyle{plain}
\bibliography{refs}
\end{document}